\documentclass[12pt,reqno]{amsart}
\usepackage{amsfonts}
\usepackage{amssymb,amsmath}
\usepackage{amscd}
\usepackage{epsfig}
\usepackage{graphicx}
\usepackage{hyperref}

\newtheorem{theorem}{Theorem}[section]

\newtheorem{corollary}{Corollary}[section]

\newtheorem{example}{Example}[section]

\numberwithin{equation}{section}

\begin{document}
\title[On the Zeros of $R$-Bonacci Polynomials]{On the Zeros of $R$-Bonacci
Polynomials and Their Derivatives}
\author[N. YILMAZ \"{O}ZG\"{U}R ]{N\.{I}HAL YILMAZ \"{O}ZG\"{U}R}
\address{Bal\i kesir University\\
Department of Mathematics\\
10145 Bal\i kesir, TURKEY}
\email{nihal@balikesir.edu.tr}
\author[\"{O}. \"{O}ZTUN\c{C} ]{\"{O}ZNUR \"{O}ZTUN\c{C}}
\address{Bal\i kesir University\\
10145 Bal\i kesir, TURKEY}
\email{oztunc@balikesir.edu.tr}
\date{}

\begin{abstract}
The purpose of the present paper is to examine the zeros of $R$-Bonacci
polynomials and their derivatives. We confirm a conjecture about the zeros
of $R$-Bonacci polynomials for some special cases. We also find explicit
formulas of the roots of derivatives of $R$-Bonacci polynomials in some
special cases.
\end{abstract}

\maketitle

\let\thefootnote\relax\footnote{%
Both authors are supported by the Scientific Research Projects Unit of Bal\i
kesir University under the project number Mat.BAP.2013.0001.}

\section{Introduction}

\label{sec:1}

The problem finding a convenient method to determine the zeros of a
polynomial has a long history that dates back to the work of Cauchy \cite%
{Marden}. In this paper our aim is to examine the zeros of $R$-Bonacci
polynomials and their derivatives. $R$-Bonacci polynomials $R_{n}(x)$ are
defined by the following recursive equation in \cite{Hogg} for any integer $n
$ and $r\geq 2:$
\begin{equation}
R_{n+r}(x)=x^{r-1}R_{n+r-1}\left( x\right) +x^{r-2}R_{n+r-2}\left( x\right)
+...+R_{n}\left( x\right) \text{,}  \label{definition}
\end{equation}%
with the initial values$\ R_{-k}(x)=0,\ k=0,1,...,r-2,\ R_{1}(x)=1,\
R_{2}\left( x\right) =x^{r-1}$. For $r=2$ we obtain the classical Fibonacci
polynomials. There are a great number of publications regarding to Fibonacci
polynomials and their generalizations, (see \cite{Ricci}-\cite{Hogg}, \cite{Koshy} and \cite{oztunc}). In\textit{\ }\cite{Hogg2},\textit{\ }V. E.\textit{\ }Hoggat
and M. Bicknell found the zeros of these polynomials using hyperbolic
trigonometric functions. For $r=3$ we obtain Tribonacci polynomials. Since
the open expressions are not found for the zeros of Tribonacci polynomials
and their derivatives, numerical studies have been studied in recent years.
Zero attractors of these polynomials are obtained by W. Goh, M. X. He and P.
E. Ricci in \cite{zero}. In \cite{Cof}, the number of the real roots of
Tribonacci-coefficient polynomials are found.

The symmetric polynomials of the zeros of\textit{\ }Fibonacci and Tribonacci
are found by M. X. He, D. Simon and P. E. Ricci in \cite{Ricci}.
Furthermore, in \cite{num}, it was determined the location and distribution
of the zeros of the Fibonacci and Tribonacci polynomials. They found out
interesting geometric properties of these polynomials. They proved the
following equation%
\begin{equation}
R_{n}(xw^{k})=w^{kn}R_{n}(x),w=e^{\frac{2\pi i}{r}},k=0,1,...,r-1,n=1,2,...
\label{eqn1}
\end{equation}%
and deduce that the zeros of $R$-Bonacci polynomials lie on the equally
spaced $r$-stars with respect to the argument $\frac{2\pi }{r}$. Also they
conjectured that these $r$-stars have $r-1$ branches starting at the zeros
of the equation $x^{r}+1=0$.

In Section \ref{sec:2} we confirm this conjecture for some classes of $R$%
-Bonacci polynomials. To do this we find the symmetric polynomials which are
made up of the $r^{th}\ $order of the zeros of $R$-Bonacci polynomials.
Using these symmetric polynomials, we determine the reference roots for the
polynomials $R_{rn+p}(x)$ for $p=0,1$ and $n=1$.

There have been several papers on the derivatives of the Fibonacci
polynomials (see \cite{Fil}, \cite{Fil2} and \cite{Wang}). In Section \ref%
{sec:3} we study the roots of the derivatives of $R$-Bonacci polynomials. We
obtain the most general symmetric polynomials which are made up of the $%
r^{th}\ $order of the zeros of derivatives of $R$-Bonacci polynomials. Using
these symmetric polynomials, we find some formulas for the zeros of
derivatives of $R$-Bonacci polynomials for some values of $t$. These
formulas are substantially simple and useful.

\section{\textbf{Zeros of Some Classes of }$R$\textbf{-Bonacci Polynomials}}

\label{sec:2}

It was given the general representations for $R$-Bonacci polynomials as \cite%
{Hogg}%
\begin{equation}
R_{n}(x)=\underset{j=0}{\overset{\left[ \frac{(r-1)(n-1)}{r}\right] }{\sum }}%
\binom{n-j-1}{j}_{r}x^{(r-1)(n-1)-rj}.  \label{algebraic presentation}
\end{equation}%
Here $\binom{n}{j}_{r}$denotes the $r$-nomial coefficient. In this section,
we obtain the symmetric polynomials including the zeros of $R$-Bonacci
polynomials. Our results are coincide with the ones obtained in \cite{Ricci}
for $R=2,3$.

For the definition of a symmetric polynomial see \cite{Ricci}. Let $\left\{
x_{1},...,x_{r}\right\} $ be the set of the reference zeros of the
polynomial $R_{rn}\left( x\right) $.

\begin{theorem}
\label{theorem1} The most general form of the $j^{th}$symmetric polynomials
consisting of over the $r^{th}$ of zeros of $R_{rn}\left( x\right) \ $is as
follows$:$%
\begin{equation}
{\small \sigma }_{j}\left( x_{1}^{r},...,x_{(r-1)n-1}^{r}\right) {\small =}%
(-1)^{j}\binom{rn-j-1}{j}_{r}{\small .}  \label{eq20}
\end{equation}
\end{theorem}

\begin{proof}
By (\ref{eqn1}), the zeros of $R$-Bonacci polynomials lie in the argument $%
\frac{2\pi }{r}$ and hence the polynomial ${\small R}_{rn}\left( x\right) $
can be factorized as
\begin{equation*}
{\small R}_{rn}\left( x\right) {\small =x}\underset{k=1}{\overset{(r-1)n-1}{%
\prod }}\left( x-x_{k}\right) \left( x-x_{k}e^{\frac{2\pi i}{r}}\right)
...\left( x-x_{k}e^{-\frac{2\pi i}{r}}\right) .
\end{equation*}%
If we rearrange this equation we obtain
\begin{eqnarray}
R_{rn}\left( x\right) &=&x\{x^{r^{2}n-rn-r}-  \notag \\
&&x^{r^{2}n-rn-2r}\underset{k=1}{\overset{(r-1)n-1}{\sum }}%
x_{k}^{r}+x^{r^{2}n-rn-3r}\underset{j\neq k}{\overset{}{\sum }}%
x_{j}^{r}x_{k}^{r}  \notag \\
&&-x^{r^{2}n-rn-4r}\underset{j\neq k\neq l}{\overset{}{\sum }}%
x_{j}^{r}x_{k}^{r}x_{l}^{r}+...-\underset{k=1}{\overset{(r-1)n-1}{\prod }}%
x_{k}^{r}\}  \notag \\
&=&\left\{ \underset{j=0}{\overset{(r-1)n-1}{\sum }}%
(-1)^{j}x^{(r-1)(rn-1)-rj}\left\{ \underset{1=l_{1}<l_{2}<...<l_{j}}{\overset%
{}{\sum }}\underset{i=1}{\overset{j}{\prod }}x_{l_{i}}^{r}\right\} \right\}
\notag \\
&=&\underset{j=0}{\overset{(r-1)n-1}{\sum }}(-1)^{j}\sigma _{j}\left(
x_{1}^{r},x_{2}^{r},...,x_{(r-1)n-1}^{r}\right) x^{(r-1)(rn-1)-rj}.
\label{eqn25}
\end{eqnarray}

On the other hand by (\ref{algebraic presentation}) we can write
\begin{equation}
R_{rn}(x)=\underset{j=0}{\overset{(r-1)n-1}{\sum }}\binom{rn-j-1}{j}%
_{r}x^{(r-1)(rn-1)-rj}.  \label{eqn26}
\end{equation}%
Since the equations $($\ref{eqn25}$)$ and $($\ref{eqn26}$)$ are equal, we
obtain the desired result $\left( \text{\ref{eq20}}\right) $.
\end{proof}

\begin{corollary}
\label{cor2}The following equations are satisfied by the zeros of $%
R_{rn}\left( x\right) :$
\begin{equation}
\underset{k=1}{\overset{(r-1)n-1}{\sum }}x_{k}^{r}=-\binom{rn-2}{1}_{r}%
{\tiny .}  \label{eq27}
\end{equation}

\begin{proof}
By setting $j=1$ in the equation $($\ref{eq20}$)$ desired result is obtained.
\end{proof}
\end{corollary}

\begin{theorem}
\label{thm2}The most general form of the $j^{th}$ symmetric polynomials
consisting of the $r^{th}$ zeros of $R_{rn+1}\left( x\right) \ $is as
follows $:$%
\begin{equation}
{\small \sigma }_{j}\left( x_{1}^{r},...,x_{(r-1)n}^{r}\right) {\small =}%
(-1)^{j}\binom{rn-j}{j}_{r}{\small .}  \label{eq76}
\end{equation}
\end{theorem}

\begin{proof}
By a similar way used in the proof of Theorem \ref{theorem1} we can write
\begin{equation*}
{\small R}_{rn+1}\left( x\right) {\small =}\underset{k=1}{\overset{(r-1)n}{%
\prod }}\left( x-x_{k}\right) \left( x-x_{k}e^{\frac{2\pi i}{r}}\right)
...\left( x-x_{k}e^{-\frac{2\pi i}{r}}\right) .
\end{equation*}%
Then we get
\begin{eqnarray}
R_{rn+1}\left( x\right) &=&\{x^{r^{2}n-rn}-  \notag \\
&&x^{r^{2}n-rn-r}\underset{k=1}{\overset{(r-1)n}{\sum }}%
x_{k}^{r}+x^{r^{2}n-rn-2r}\underset{j\neq k}{\overset{}{\sum }}%
x_{j}^{r}x_{k}^{r}  \notag \\
&&-x^{r^{2}n-rn-3r}\underset{j\neq k\neq l}{\overset{}{\sum }}%
x_{j}^{r}x_{k}^{r}x_{l}^{r}+...-\underset{k=1}{\overset{(r-1)n}{\prod }}%
x_{k}^{r}\}  \notag \\
&=&\left\{ \underset{j=0}{\overset{(r-1)n}{\sum }}(-1)^{j}x^{rn(r-1)-rj}%
\left\{ \underset{1=l_{1}<l_{2}<...<l_{j}}{\overset{}{\sum }}\underset{i=1}{%
\overset{j}{\prod }}x_{l_{i}}^{r}\right\} \right\}  \notag \\
&=&\underset{j=0}{\overset{(r-1)n}{\sum }}(-1)^{j}\ \sigma _{j}\left(
x_{1}^{r},x_{2}^{r},...,x_{(r-1)n}^{r}\right) x^{rn(r-1)-rj}.  \label{eq75}
\end{eqnarray}%
By putting $rn+1$ instead of $n$ in (\ref{algebraic presentation}) we find%
\begin{equation}
R_{rn+1}\left( x\right) =\underset{j=0}{\overset{n(r-1)}{\sum }}\binom{rn-j}{%
j}_{r}x^{(r-1)rn-rj}.  \label{eq77}
\end{equation}

It follows from the comparison $($\ref{eq75}$)$ and $($\ref{eq77}$)$ \ it is
possible to write the desired result $\left( \text{\ref{eq76}}\right) $.
\end{proof}

\begin{corollary}
\label{cor3} The following equations are satisfied by the zeros of $%
R_{rn+1}\left( x\right) :$
\begin{equation}
\underset{k=1}{\overset{(r-1)n}{\sum }}x_{k}^{r}=-\binom{rn-1}{1}_{r}{\tiny .%
}  \label{eq78}
\end{equation}

\begin{proof}
If we set $j=1$ in the equation $($\ref{eq76}$)$ then we get the equation $($%
\ref{eq78}$)$.
\end{proof}
\end{corollary}

Now using these symmetric polynomials we obtain the reference roots of $%
R_{rn+p}(x)$ for $p=0,1$.

\begin{theorem}
\label{thm8} For $p=0,1$ and $n=1,$ let $x_{j}(1\leq j\leq r)$ be the
reference zeros of $R_{rn+p}\left( x\right) .$ Then we have
\begin{equation}
x_{j}^{r}=-1.  \label{eq38}
\end{equation}
\end{theorem}

\begin{proof}
Let $p=0$ or $p=1$ and let the set of the reference zeros of $R_{rn+p}\left(
x\right) $ be $\left\{ x_{1},...,x_{r}\right\} $. The other zeros of the
polynomial $R_{rn+p}\left( x\right) $ will be generated by the argument $%
\frac{2\pi }{r}$ except the root $x=0$. For a fixed $j$, using the equations
$($\ref{eq78}$)$ and $($\ref{eq27}$)$ we have
\begin{equation*}
\underset{k=1}{\overset{r-1}{\sum }}%
x_{k}^{r}=x_{1}^{r}+x_{2}^{r}+...+x_{r-1}^{r}=x_{j}^{r}+\left( x_{j}e^{\frac{%
2\pi i}{r}}\right) ^{r}+\left( x_{j}e^{\frac{4\pi i}{r}}\right)
^{r}+...+\left( x_{j}e^{\frac{2(r-2)\pi i}{r}}\right) ^{r}=-(r-1)
\end{equation*}%
and%
\begin{equation*}
\underset{k=1}{\overset{r-2}{\sum }}%
x_{k}^{r}=x_{1}^{r}+x_{2}^{r}+...+x_{r-1}^{r}=x_{j}^{r}+\left( x_{j}\ e^{%
\frac{2\pi i}{r}}\right) ^{r}+\left( x_{j}\ e^{\frac{4\pi i}{r}}\right)
^{r}+...+\left( x_{j}e^{\frac{2(r-3)\pi i}{r}}\right) ^{r}=-(r-2),
\end{equation*}%
respectively. Rearranging the above equations, it can be easily seen that
the reference roots of $R_{rn+p}\left( x\right) $ as in the equation $($\ref%
{eq38}$).$
\end{proof}

\begin{figure}[t]
\centering
\includegraphics[width=6.5cm]{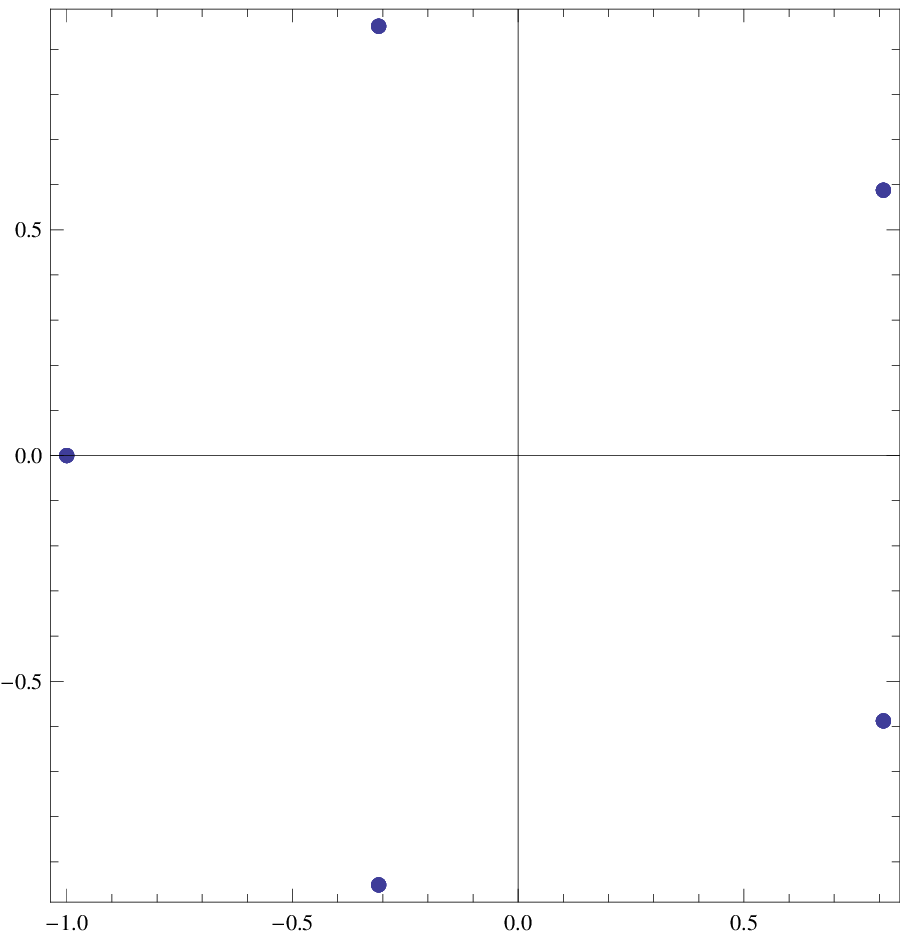}
\caption{The zeros of $B_{6}\left( x\right)$}
\label{fig:1}
\end{figure}

\begin{example}
Let us consider the following $5$-Bonacci polynomial
\begin{equation*}
B_{6}\left( x\right) =(x^{5}+1)^{4}.
\end{equation*}%
Using $\left( \text{\ref{eq38}}\right) $ if we solve the equation $%
x_{j}^{5}=-1(1\leq j\leq 5)$, the reference roots of the polynomial $%
B_{6}\left( x\right) $ are found as follows $($see Figure \ref{fig:1}$):$%
\begin{equation*}
x_{1}=\left( -1\right) ,x_{2}=\left( -1\right) ^{\frac{1}{5}},x_{3}=\left(
-1\right) ^{\frac{2}{5}},x_{4}=\left( -1\right) ^{\frac{3}{5}},x_{5}=\left(
-1\right) ^{\frac{4}{5}}.
\end{equation*}
\end{example}

\section{\textbf{Zeros of Derivatives of }$R$\textbf{-Bonacci Polynomials}}

\label{sec:3}

Before we find the symmetric polynomials which are made up of the $r^{th}$
order of the zeros of the derivatives of $R$-Bonacci polynomials $%
R_{n}^{(t)}(x)$ we write the algebraic representations of them. For any
fixed $n$, using the equation $\left( \text{\ref{algebraic presentation}}%
\right) $, the algebraic representation of the derivative polynomial $%
R_{n}^{(t)}(x)$ is obtained%
\begin{equation}
R_{n}^{(t)}(x)=\underset{j=0}{\overset{}{\sum }}\binom{n-j-1}{j}%
_{r}((r-1)(n-1)-rj)...((r-1)(n-1)-rj-t+1)x^{(r-1)(n-1)-rj-t}.  \label{eq21}
\end{equation}

Now we determine the symmetric polynomials for $R_{rn+p}^{(t)}(x)\ $for
particular values of $t$. We give the following theorem.

\begin{theorem}
\label{thm4} Let $k\in
%TCIMACRO{\U{2115} }%
%BeginExpansion
\mathbb{N}
%EndExpansion
^{+}$, $p\in \left\{ 0,1,...,r-1\right\} $. If we define
\begin{equation}
t=rk-(1-p)(r-1),  \label{eqn21}
\end{equation}%
\begin{equation}
\mu =((r-1)(rn+p-1))...(rn(r-1)-t+(p-1)r+(2-p))  \label{eqn22}
\end{equation}%
and%
\begin{equation}
\eta =(r-1)n-\left( \frac{t+(1-p)(r-1)}{r}\right) \text{,}  \label{eqn23}
\end{equation}%
then$\ $the most general form of the symmetric polynomials consisting of the
zeros of $R_{rn+p}^{(t)}\left( x\right) \ $is as follows$:$%
\begin{equation}
{\small \sigma }\left( x_{1}^{r},...,x_{\eta }^{r}\right) {\small =}\frac{%
(-1)^{j}((r-1)(rn+p-1)-rj)...((r-1)(rn+p-1)-rj-t+1)}{\mu }\binom{rn+p-j-1}{j}%
_{r}{\small .}  \label{eqn24}
\end{equation}
\end{theorem}

\begin{proof}
It can be easily seen that%
\begin{equation*}
{\small R}_{rn+p}^{(t)}\left( x\right) {\small =}\mu \underset{k=1}{\overset{%
\eta }{\prod }}\left( x-x_{k}\right) \left( x-x_{k}e^{\frac{2\pi i}{r}%
}\right) ...\left( x-x_{k}e^{-\frac{2\pi i}{r}}\right) ,
\end{equation*}%
where $\mu $ is a coefficient. Then we have
\begin{eqnarray}
R_{rn+p}^{(t)}\left( x\right) &=&\mu \{x^{r^{2}n-rn-(t+(1-p)(r-1))}-  \notag
\\
&&x^{r^{2}n-rn-(t+(1-p)(r-1))-r}\underset{k=1}{\overset{\eta }{\sum }}%
x_{k}^{r}+x^{r^{2}n-rn-(t+(1-p)(r-1))-2r}\underset{j\neq k}{\overset{}{\sum }%
}x_{j}^{r}\ x_{k}^{r}  \notag \\
&&-x^{r^{2}n-rn-(t+(1-p)(r-1))-3r}\underset{j\neq k\neq l}{\overset{}{\sum }}%
x_{j}^{r}x_{k}^{r}x_{l}^{r}+...-\underset{k=1}{\overset{\eta }{\prod }}%
x_{k}^{r}\}  \notag \\
&=&\mu \left\{ \underset{j=0}{\overset{\eta }{\sum }}%
(-1)^{j}x^{r^{2}n-rn-(t+(1-p)(r-1))-rj}\left\{ \underset{%
1=l_{1}<l_{2}<...<l_{j}}{\overset{}{\sum }}\underset{i=1}{\overset{j}{\prod }%
}x_{l_{i}}^{r}\right\} \right\}  \notag \\
&=&\mu \underset{j=0}{\overset{\eta }{\sum }}(-1)^{j}\sigma _{j}\left(
x_{1}^{r},x_{2}^{r},...,x_{\eta }^{r}\right) x^{(r-1)(rn+p-1)-rj-t}.
\label{eq24}
\end{eqnarray}%
By using the equation (\ref{eq21}) and taking $rn+p$ instead of $n$ we can
write%
\begin{equation}
R_{rn+p}^{(t)}(x)=\underset{j=0}{\overset{}{\sum }}\binom{rn+p-j-1}{j}%
_{r}((r-1)(rn+p-1)-rj)...((r-1)(rn+p-1)-rj-t+1)x^{(r-1)(rn+p-1)-rj-t}.
\label{eq25}
\end{equation}

Since the equations $($\ref{eq24}$)$ and $($\ref{eq25}$)$ are equal the
proof follows.
\end{proof}

\begin{corollary}
Let $t$ and $\eta $ be as in the equations $($\ref{eqn21}$)$ and $($\ref%
{eqn23}$)$, respectively. For $k\in
%TCIMACRO{\U{2115} }%
%BeginExpansion
\mathbb{N}
%EndExpansion
^{+}$ and $p\in \left\{ 0,1,..,r-1\right\} $, the following equations are
satisfied by the zeros of $R_{rn+p}^{(t)}\left( x\right) :$
\begin{equation}
(i)\underset{k=1}{\overset{\eta }{\prod }}x_{k}^{r}=\frac{(-1)^{\eta }t\
(t-1)...(1)}{((r-1)(rn+p-1))...(rn(r-1)-t+(p-1)r+(2-p))}\binom{rn+p-\eta -1}{%
\eta }_{r}{\tiny .}  \label{eq26}
\end{equation}%
and%
\begin{equation}
(ii)\underset{k=1}{\overset{\eta }{\sum }}x_{k}^{r}=-\frac{%
((r-1)(rn+p-1)-r)...((r-1)(rn+p-1)-r-t+1)}{%
((r-1)(rn+p-1))...(rn(r-1)-t+(p-1)r+(2-p))}\binom{rn+p-2}{1}_{r}{\tiny .}
\label{eq28}
\end{equation}

\begin{proof}
In the equation $($\ref{eqn24}$)$, if we put $j=\eta \ $and $j=1$ we obtain
the desired results, respectively.
\end{proof}
\end{corollary}

Let
\begin{equation}
\upsilon _{\eta }=\frac{(-1)^{\eta }t\ (t-1)...(1)}{%
((r-1)(rn+p-1))...(rn(r-1)-t+(p-1)r+(2-p))}\binom{rn+p-\eta -1}{\eta }_{r}
\label{eqn104}
\end{equation}%
and
\begin{equation}
\psi _{\eta }=-\frac{((r-1)(rn+p-1)-r)...((r-1)(rn+p-1)-r-t+1)}{%
((r-1)(rn+p-1))...(rn(r-1)-t+(p-1)r+(2-p))}\binom{rn+p-2}{1}_{r}{\tiny .}
\label{eqn105}
\end{equation}%
Then we can give the following theorem.

\begin{theorem}
For $t=r(r-1)n-2r-(1-p)(r-1)$, $R_{rn+p}^{(t)}\left( x\right) $ has $\left[
r((r-1)n-\left( \frac{t+(1-p)(r-1)}{r}\right) )\right] $ roots and these
roots are
\begin{equation}
x_{k}=\left( \frac{\psi _{2}\pm \sqrt{\psi _{2}^{2}-4\upsilon _{2}}}{2}%
\right) ^{\frac{1}{r}}e^{\frac{2k\pi i}{r}},(k=0,1,...,r-1),  \label{eqno1}
\end{equation}%
where $\upsilon _{2}$ and $\psi _{2}$ are defined by the equations $($\ref%
{eqn104}$)$ and $($\ref{eqn105}$)$, respectively.
\end{theorem}

\begin{proof}
Since $R_{rn+p}^{(r(r-1)n-2r-(1-p)(r-1))}\left( x\right) $ is a polynomial
of $r((r-1)n-\left( \frac{t+(1-p)(r-1)}{r}\right) )$-th degree then by using
the equations $($\ref{eq26}$)$ and $($\ref{eq28}$)$ we have
\begin{equation}
\underset{k=1}{\overset{2}{\prod }}x_{k}^{r}=x_{1}^{r}x_{2}^{r}=\upsilon _{2}
\label{eqno2}
\end{equation}%
and%
\begin{equation}
\underset{k=1}{\overset{2}{\sum }}x_{k}^{r}=x_{1}^{r}+x_{2}^{r}=\psi _{2}%
\text{.}  \label{eqno4}
\end{equation}

Since we know that $x_{1}^{r}=\frac{\upsilon _{2}}{x_{2}^{r}}$ it can be
easily seen that
\begin{equation*}
x_{2}^{2r}-\psi _{2}\ x_{2}^{r}+\upsilon _{2}=0.
\end{equation*}%
Solving this last equation of the second degree, the roots can be easily
found. So the roots of $R_{rn+p}^{(t)}\left( x\right) \ $must be as the
equation $($\ref{eqno1}$)$.
\end{proof}

Since we have Fibonacci and Tribonacci polynomials for $r=2$ and $r=3$,
respectively, we can give the following corollaries.

\begin{corollary}
Let $p\in \left\{ 0,1\right\} $ and $t=2n-5+p$. The zeros of the polynomial $%
F_{2n+p}^{(t)}\left( x\right) $ can be formulized as follows$:$%
\begin{equation*}
x_{k}=\left( \frac{\psi _{2}\pm \sqrt{\psi _{2}^{2}-4\upsilon _{2}}}{2}%
\right) ^{\frac{1}{2}}e^{k\pi i},(k=0,1)
\end{equation*}%
where $\upsilon _{2}$ and $\psi _{2}$ are defined by the equations $($\ref%
{eqn104}$)$ and $($\ref{eqn105}$)$, respectively.
\end{corollary}

In \cite{Wang}, J. Wang proved the following equation for any fixed $n$
\begin{equation}
L_{n}^{^{(t)}}\left( x\right) =nF_{n}^{(t-1)}\left( x\right) ,n\geq 1\text{,}
\label{eqno3}
\end{equation}%
where $L_{n}\left( x\right) $ are Lucas polynomials. Hence the zeros of $%
L_{n}^{^{(t+1)}}\left( x\right) $ and $F_{n}^{(t)}\left( x\right) $ are
identical.

\begin{corollary}
Let $p\in \left\{ 0,1,2\right\} $ and $t=6n-8+2p$. The zeros of the
polynomial $T_{3n+p}^{(t)}\left( x\right) $ are%
\begin{equation}
x_{k}=\left( \frac{\psi _{2}\pm \sqrt{\psi _{2}^{2}-4\upsilon _{2}}}{2}%
\right) ^{\frac{1}{3}}e^{\frac{2k\pi i}{3}}(k=0,1,2),  \label{eqno30}
\end{equation}%
where $\upsilon _{2}$ and $\psi _{2}$ are defined by the equations $($\ref%
{eqn104}$)$ and $($\ref{eqn105}$)$, respectively.
\end{corollary}

Now we give some examples.
\begin{figure}[t]
\centering
\includegraphics[width=6.5cm]{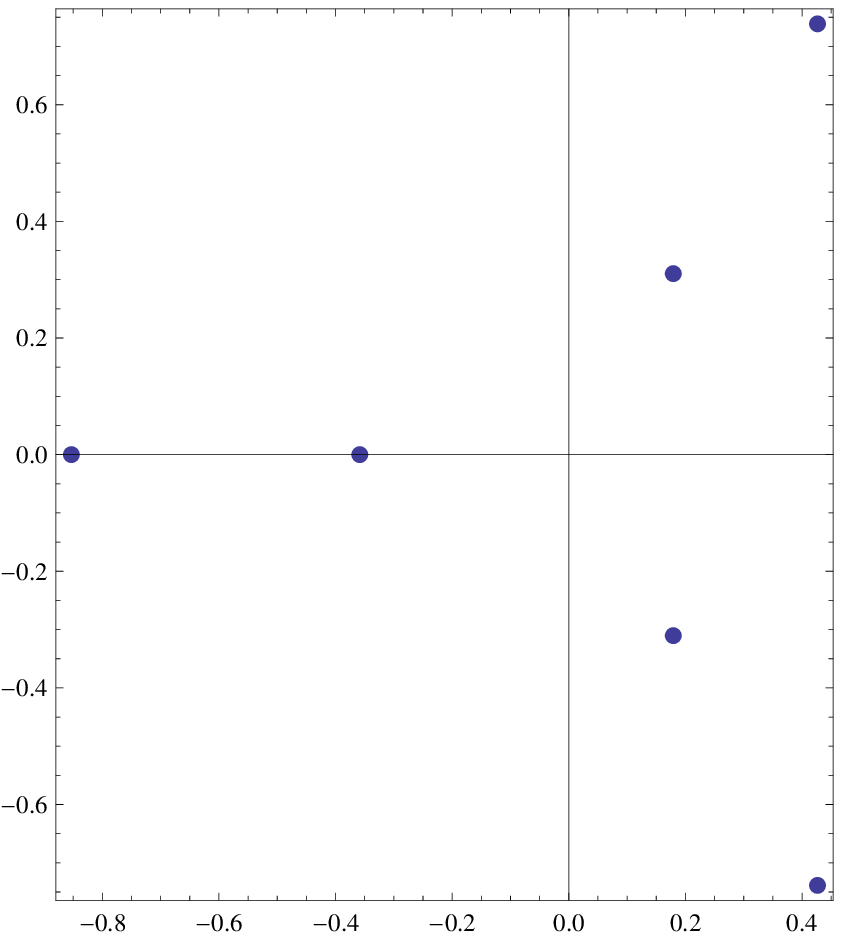}
\caption{The roots of $T_{6}^{(\imath v)}\left( x\right) $.}
\label{fig:2}
\end{figure}

\begin{example}
We consider the zeros of the polynomial
\begin{equation*}
T_{6}^{(\imath v)}\left( x\right) =5040x^{6}+3360x^{3}+144\text{.}
\end{equation*}%
In the equation $($\ref{eqno30}$),\ $writing $\psi _{2}=2/3$, $\upsilon
_{2}=1/35$, we find the zeros of this polynomial as%
\begin{equation*}
x_{k}=\sqrt[3]{\frac{2/3\pm \sqrt{\left( 2/3\right) ^{2}-4/35}}{2}}e^{\frac{%
2k\pi i}{3}},(k=0,1,2)
\end{equation*}%
$($see Figure \ref{fig:2}$)$.
\end{example}

\begin{figure}[t]
\centering
\includegraphics[width=6.5cm]{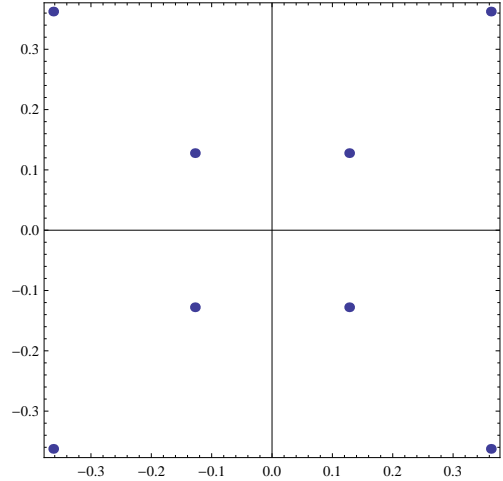}
\caption{The roots of $Q_{8}^{(13)}(x)$}
\label{fig:3}
\end{figure}

\begin{example}
For $p=0$,$\ n=2$ and $r=4$, let us consider the polynomial
\begin{equation*}
Q_{8}^{(13)}\left( x\right)
=93405312000+88921857024000x^{4}+1267136462592000x^{8}.
\end{equation*}

Using the equations $($\ref{eq26}$)$ and $($\ref{eq28}$)$ we have
\begin{equation*}
\underset{k=1}{\overset{2}{\prod }}x_{k}^{4}=\frac{1}{13566}=\upsilon _{2}
\end{equation*}%
and
\begin{equation*}
\underset{k=1}{\overset{2}{\sum }}x_{k}^{4}=-\frac{4}{57}=\psi _{2}.
\end{equation*}%
\ Then the roots of $Q_{8}^{(13)}\left( x\right) $ are generated by $x_{k}$ $%
(k=0,1,2,3)$. By $($\ref{eqno1}$)$, the roots of the polynomial $%
Q_{8}^{(13)}\left( x\right) $ are obtained as%
\begin{equation*}
x_{1}=\sqrt[4]{\frac{-\frac{4}{57}+\sqrt{(-\frac{4}{57})^{2}-\frac{4}{13566}}%
}{2}}=0.127788+0.127788i,
\end{equation*}%
and
\begin{equation*}
x_{2}=\sqrt[4]{\frac{-\frac{4}{57}-\sqrt{(-\frac{4}{57})^{2}-\frac{4}{13566}}%
}{2}}=0.36255+0.36255i
\end{equation*}%
for $k=0$,%
\begin{equation*}
x_{3}=\sqrt[4]{\frac{-\frac{4}{57}+\sqrt{(-\frac{4}{57})^{2}-\frac{4}{13566}}%
}{2}}e^{\frac{\pi i}{2}}=-0.36255+0.36255i
\end{equation*}%
and%
\begin{equation*}
x_{4}=\sqrt[4]{\frac{-\frac{4}{57}-\sqrt{(-\frac{4}{57})^{2}-\frac{4}{13566}}%
}{2}e^{\frac{\pi i}{2}}}=-0.127788+0.127788i
\end{equation*}%
for $k=1$,%
\begin{equation*}
x_{5}=\sqrt[4]{\frac{-\frac{4}{57}+\sqrt{(-\frac{4}{57})^{2}-\frac{4}{13566}}%
}{2}\text{ }}e^{\pi i}=-0.127788-0.127788i
\end{equation*}%
and
\begin{equation*}
x_{6}=\sqrt[4]{\frac{-\frac{4}{57}-\sqrt{(-\frac{4}{57})^{2}-\frac{4}{13566}}%
}{2}e^{\pi i}}=-0.36255-0.36255i,
\end{equation*}%
for$\ k=2,$\
\begin{equation*}
x_{7}=\sqrt[4]{\frac{-\frac{4}{57}+\sqrt{(-\frac{4}{57})^{2}-\frac{4}{13566}}%
}{2}\text{ }}e^{\frac{3\pi i}{2}}=0.127788-0.127788i
\end{equation*}%
and%
\begin{equation*}
x_{8}=\sqrt[4]{\frac{-\frac{4}{57}-\sqrt{(-\frac{4}{57})^{2}-\frac{4}{13566}}%
}{2}e^{\frac{3\pi i}{2}}}=0.36255-0.36255i,
\end{equation*}%
$\ $for$\ k=3$ $($see Figure \ref{fig:3}$)$.
\end{example}

\begin{figure}[t]
\centering
\includegraphics[width=6.5cm]{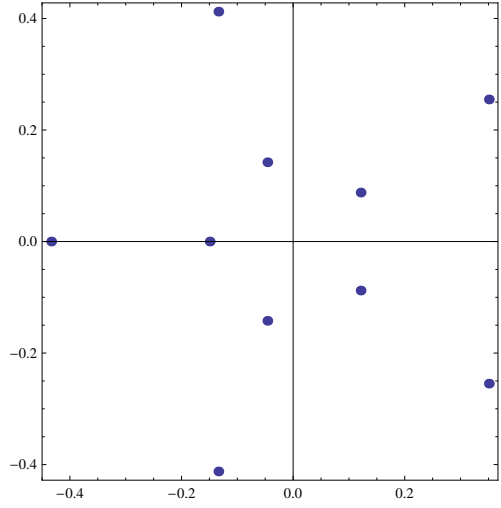}
\caption{The roots of $B_{8}^{(18)}(x)$}
\label{fig:4}
\end{figure}

\begin{example}
Let us consider the $5$-Bonacci polynomials $B_{8}^{18}(x)$. In this case we
have $p=3$,$\ n=1$ and $r=5$ and we obtain that
\begin{equation*}
B_{8}^{18}(x)=96035605585920000+1292600836944248832000x^{5}+84019054401376174080000x^{10}%
\text{.}
\end{equation*}%
The roots of this polynomial are found as follows $($see Figure \ref{fig:4}$%
):$%
\begin{equation*}
x_{k}=\sqrt[5]{\frac{\psi _{2}\pm \sqrt{\psi _{2}^{2}-4\upsilon _{2}}}{2}}e^{%
\frac{2k\pi i}{5}},k=0,1,2,3,4\text{.}
\end{equation*}
\end{example}

%%%% Bibliography  %%%%%%%%%%

\end{document}